\documentclass[a4paper,UKenglish,cleveref,hyperref,autoref]{lipics-v2021}

\title{Sparse Induced Subgraphs of Large Treewidth}
\titlerunning{Sparse Induced Subgraphs of Large Treewidth}

\author{\'{E}douard Bonnet}{Univ Lyon, CNRS, ENS de Lyon, Université Claude Bernard Lyon 1, LIP UMR5668, France \and \url{http://perso.ens-lyon.fr/edouard.bonnet/}}{edouard.bonnet@ens-lyon.fr}{https://orcid.org/0000-0002-1653-5822}{}

\authorrunning{\'E. Bonnet}

\Copyright{Édouard Bonnet}

\category{}

\relatedversion{}

\supplement{}

\acknowledgements{We thank Maria Chudnovsky and Nicolas Trotignon for useful discussions.}

\nolinenumbers 

\hideLIPIcs  

\EventEditors{John Q. Open and Joan R. Access}
\EventNoEds{2}
\EventLongTitle{42nd Conference on Very Important Topics (CVIT 2016)}
\EventShortTitle{CVIT 2016}
\EventAcronym{CVIT}
\EventYear{2016}
\EventDate{December 24--27, 2016}
\EventLocation{Little Whinging, United Kingdom}
\EventLogo{}
\SeriesVolume{42}
\ArticleNo{23}

\usepackage[utf8]{inputenc}  

\usepackage[T1]{fontenc}
\usepackage{lmodern}

\usepackage[colorinlistoftodos,bordercolor=orange,backgroundcolor=orange!20,linecolor=orange,textsize=normalsize]{todonotes}

\usepackage{amsmath}  
\usepackage{amssymb}     
\usepackage{bbm}
\usepackage{accents}
\usepackage{complexity}

\usepackage{booktabs}
\usepackage{paralist}
\usepackage{fixmath}

\makeatletter
\newtheorem*{rep@theorem}{\rep@title}
\newcommand{\newreptheorem}[2]{%
\newenvironment{rep#1}[1]{%
 \def\rep@title{#2 \ref{##1}}%
 \begin{rep@theorem}}%
 {\end{rep@theorem}}}
\makeatother

\newreptheorem{theorem}{Theorem}
\newreptheorem{lemma}{Lemma}

\usepackage{xspace}
\usepackage{tikz}

\usepackage[ruled,vlined,linesnumbered]{algorithm2e}

\usetikzlibrary{fit}
\usetikzlibrary{arrows}
\usetikzlibrary{patterns}
\usetikzlibrary{calc}
\usetikzlibrary{shapes}
\usetikzlibrary{positioning}
\usetikzlibrary{math}
\usetikzlibrary{shapes.symbols}
\usetikzlibrary{decorations.pathreplacing,calligraphy}
\usepackage[scr=boondox,scrscaled=1.05]{mathalfa}

\newcommand{\card}[1]{|{#1}|}

\newcommand\tw{\text{tw}}
\newcommand\bn{\text{bn}}

\begin{document}

\maketitle

\begin{abstract}
  Motivated by an induced counterpart of treewidth sparsifiers (i.e., sparse subgraphs keeping the treewidth large) provided by the celebrated Grid Minor theorem of Robertson and Seymour [JCTB '86] or by a~classic result of Chekuri and Chuzhoy [SODA '15], we show that for any natural numbers $t$ and $w$, and real $\varepsilon > 0$, there is an integer $W := W(t,w,\varepsilon)$ such that every graph with treewidth at least $W$ and no $K_{t,t}$ subgraph admits a 2-connected $n$-vertex induced subgraph with treewidth at~least $w$ and at~most $(1+\varepsilon)n$ edges.
  The induced subgraph is either a subdivided wall, or its line graph, or a spanning supergraph of a subdivided biclique.
  This in particular extends a result of Weißauer [JCTB '19] that graphs of large treewidth have a large biclique subgraph or a long induced cycle.
\end{abstract}

\section{Introduction}\label{sec:intro}

The celebrated Grid Minor theorem~\cite{RobertsonS86} implies that graphs with high treewidth admit large subdivided walls as subgraphs.
Thus, in particular, they contain subcubic subgraphs of large treewidth.
A~classic result by Chekuri and Chuzhoy~\cite{Chekuri15} qualitatively improves the latter fact:
Every graph of treewidth at least $k$ contains a~subcubic subgraph with treewidth $k/\log^{\Theta(1)} k$.
These results are sometimes called ``treewidth sparsifiers'' since they provide a~sparse \emph{subgraph} retaining large treewidth.
What about sparse \emph{induced subgraphs} retaining large treewidth?

As cliques and bicliques have large treewidth but no induced subgraphs of large treewidth that are sparse, for any standard definition of \emph{sparse}, we have to exclude them as induced subgraphs.
Equivalently, we have to exclude large bicliques as subgraphs.
Even then, we cannot ask for as much ``sparsification'' as one gets with subgraphs.
There are indeed several constructions of graphs with large treewidth and no large biclique subgraph but yet no induced subgraph of both large treewidth and small maximum degree:
for instance, the so-called \emph{ttf-layered-wheels} of Sintiari and Trotignon~\cite{SintiariT21} (where \emph{ttf} stands for triangle and theta-free), or a~grid where every ``column'' is replaced by a star, due to Pohoaţă~\cite{Pohoata14} and to Davies~\cite{Davies22}, or a~construction of Bonamy et al.~\cite{Bonamy23} without two mutually induced cycles nor $K_{2,3}$ subgraph.
Furthermore, all those constructions can be made of arbitrary large girth.

While none of these families were primarily designed to avoid induced subgraphs of large treewidth and small maximum degree, it is in hindsight necessary that they do as they all exclude a~planar induced minor: $K_{2,3}$ (the \emph{theta graph}), the $5 \times 5$ grid, and $2K_3$ (two disjoint triangles), respectively.
Indeed, by a~result of Korhonen~\cite{Korhonen23}, we now know that such families (avoiding a~fixed planar graph as an induced minor) have treewidth at~most some function of their maximum degree.

Which notion of sparsity should one then pick?
Hajebi~\cite{Hajebi22,Hajebi24} suggested to go with degeneracy and asked whether weakly sparse (i.e., avoiding a~biclique subgraph) graphs of high treewidth always admit 2-degenerate induced subgraphs of large treewidth.
This is a~sensible question as a~result of K\"uhn and Osthus~\cite{Kuhn04} implies that for any natural numbers $t$ and $w$, there is some $d := d(t,w)$ such that every graph without $K_{t,t}$ subgraph is either $d$-degenerate itself or admits a~2-degenerate induced subgraph of treewidth~$w$. 
However, Hajebi's suggestion was very recently ruled out by Chudnovsky and Trotignon~\cite{Chudnovsky24}: For every natural number~$c$, there are graphs without $K_{2,2}$ as an induced subgraph, of clique number $c+1$, and of arbitrarily large treewidth such that all their $K_c$-free induced subgraphs have treewidth at most some (linear) function of~$c$.

In this paper, we go in a~somewhat different direction, and set ourselves to find an induced subgraph of large treewidth but as small average degree as possible.
Of course, one may try and ``cheat'' by decreasing the average degree of any induced subgraph $H$ of large treewidth by adding the vertices of an independent set (non-adjacent to~$H$).
To prevent that, we require the induced subgraph to be 2-connected.
This is as far as we can go, since the 1-subdivision of a~large clique has no 3-connected induced subgraphs.

Formally we show the following theorem.
We will actually prove the stronger~\cref{thm:spec-main}.

\begin{theorem}\label{thm:main}
  For any natural numbers $t$ and $w$, and real $\varepsilon > 0$, there is an integer $W := W(t,w,\varepsilon)$ such that every graph with treewidth at least $W$ and no $K_{t,t}$ subgraph admits a~2-connected $n$-vertex induced subgraph with treewidth at~least $w$ and at~most $(1+\varepsilon)n$ edges.
\end{theorem}

A~first step toward~\cref{thm:main} is the following lemma. 

\begin{lemma}\label{lem:clique-subdivision-intro}
  For any natural numbers $k$ and $s$, there is an integer $W := W(k,s)$ such that every graph of treewidth at~least $W$ either admits a~subdivision of the $s$-clique as a~subgraph or the $k \times k$ grid as an induced minor. 
\end{lemma}

\Cref{lem:clique-subdivision-intro} was previously proven under a~different formulation by Abrishami et al.~\cite[Theorem 6.5]{istd7}.
The proof of~\cref{lem:clique-subdivision-intro} mainly consists of combining results by
\begin{compactitem}
\item Korhonen~\cite{Korhonen23}, essentially showing \cref{lem:clique-subdivision-intro} with the subdivision of the $s$-clique or \emph{$K_s$~topological minor} replaced by \emph{$s$~vertices of degree at~least $s$},
\item Golovach, Fomin, and Thilikos~\cite{FominGT11}, similarly showing \cref{lem:clique-subdivision-intro} with \emph{$K_s$ topological minor} replaced by \emph{$K_s$ minor}, and
\item Grohe and Marx \cite{GroheM15} decomposing graphs excluding a~large topological minor in a~tree-decomposition of bounded adhesion whose every torso avoids either $s$ vertices of degree at least $s$ or a~$K_s$ minor.
\end{compactitem}
We give an alternative proof in our different language of induced minors and with our notations.
The only technicality in obtaining \cref{lem:clique-subdivision-intro} is that torsos do not preserve induced subgraphs nor induced minors.
We overcome it in a~different way than Abrishami et al.~\cite{istd7}, who stick to the torso, and thus need to show how to transform large induced grids in the torso into large induced grids in the original graph. 
Instead, we work with a~moral equivalent of the torso that is also an induced minor of the original graph.
On a~conceptual level, we find it simpler than the previous proof since we spare ourselves the ``transformation step.''
Our task is simply to check that Grohe and Marx's structure theorem~\cite{GroheM15} still works with our torso variant, an observation of independent interest.

As every planar graph is an induced minor of a~sufficiently large grid, the following corollary is a~slightly more compact formulation of~\cref{lem:clique-subdivision-intro}. 
Like~\cref{lem:clique-subdivision-intro}, it should be credited to Abrishami et al.~\cite{istd7}.

\begin{corollary}\label{cor:main}
  Any class of unbounded treewidth contains every planar graph as an induced minor or every graph as a~topological minor.
\end{corollary}

In passing, let us mention a~direct algorithmic application of~\cref{lem:clique-subdivision-intro}.
Let \textsc{\mbox{$H$-Induced} Minor Containment} be the problem of deciding if an input graph $G$ contains the fixed graph $H$ as an induced minor.
Korhonen and Lokshtanov~\cite{KorhonenL23} recently showed, among other things, that this problem can be NP-complete even when $H$ is a~tree.
In stark contrast, for every planar graph~$H$, the problem admits a~linear-time algorithm  on topological-minor-free classes.

\begin{corollary}\label{thm:alg-im-containment}
  Let $H$ be a~$q$-vertex planar graph, and $\mathcal C$ be a~graph class excluding at~least one fixed graph as a~topological minor. 
  There is a~computable function $f$ such that \textsc{$H$-Induced Minor Containment} can be solved in time $f(q) \cdot n$ on $n$-vertex graphs of~$\mathcal C$.
\end{corollary}
\begin{proof}
  Let $s$ be such that no graph of $\mathcal C$ has a~$K_s$ topological minor, and $k=q^{O(1)}$ such that $H$ is an induced minor of the $k \times k$ grid.
  By~\cref{lem:clique-subdivision-intro} there is some $W := W(k,s)$ such that every graph of~$\mathcal C$ with treewidth at~least $W$ is a~positive instance.

  Let $G$ be the input graph.
  One first finds a~tree-decomposition of width $2\tw(G)+1$ in time $2^{O(\tw(G))} |V(G)|$~\cite{Korhonen21}, where $\tw(G)$ denotes the treewidth of~$G$.
  If the reported tree-decomposition has width at~least $2W+1$, we correctly conclude that $G$ admits $H$ as an induced minor.
  If instead it has width at~most $2W$, we solve \textsc{$H$-Induced Minor Containment} in time $g(2W,|\varphi|) \cdot |V(G)|$ by Courcelle's theorem~\cite{Courcelle90}, where $g$ is a~computable function, and $\varphi$ is a~sentence in monadic second-order logic expressing the induced-minor containment of $H$: \emph{there are $q$~sets of vertices $S_1, \ldots, S_q$ that are pairwise disjoint and each inducing a~connected subgraph, such that there is at~least one edge between $S_i$ and $S_j$ if and only if the $i$-th and $j$-th vertices of $H$ are adjacent.} 
\end{proof}

Let us go back to the proof of our main theorem.
We prove \cref{thm:main} by first applying \cref{lem:clique-subdivision-intro}.
It is easy to find a~2-connected induced subgraph of large treewidth and edge density\footnote{Throughout the paper, we call \emph{edge density} the edge-vertex ratio (hence half of the average degree), following for instance~\cite{sparsity}. Other authors use this term for the number of edges of the graph divided by the number of edges in a~clique on the same number of vertices.} $1+\varepsilon$ (for any $\varepsilon > 0$) within the induced minor of a~large grid (see~\cref{lem:sparse-grid}).
We thus focus on what happens when we get a~large clique subdivision as a~subgraph.
Now leveraging the absence of $K_{t,t}$ subgraph, we use a~classic result by K\"uhn and Osthus~\cite{Kuhn04}, and its strengthening by Dvořák~\cite{Dvorak18}, to get as an induced subgraph a~spanning supergraph of a~long subdivision of a~large clique whose degeneracy $d$ is at~most some function of $t$ and of the treewidth target lower bound~$w$.

The clique is subdivided enough that its edge density is at~most $1+\varepsilon/2$.
But the density of the \emph{extra edges} (those that are induced by the vertices of the clique subdivision but not part of the subdivision) can be as large as~$d$.
By a~minimality property on the clique subdivision and averaging arguments, we find as an induced subgraph the topological minor of a~smaller graph with density of extra edges at~most $\varepsilon/2$, and treewidth at~least~$w$.

This requires to reduce the number of extra edges while ensuring that the number of vertices does not drop as fast.
A~natural idea is to consider random partitions of the \emph{branch vertices} (i.e., vertices of high degree within the subdivided graph) and keep only the branch vertices in one part together with the \emph{subdivision vertices} (i.e., vertices of degree 2 within the subdivision) linking (via what we call \emph{direct paths}) all the pairs of the kept branch vertices.
An issue with that plan is that the direct paths between branch vertices in distinct parts could in principle be (adversarially) longer than the average direct path exactly when the number of extra edges drops satisfactorily.
Thus, possibly, no part of the partition would actually decrease the density of extra edges.

An effective remedy is to extract biclique rather than clique subdivisions.
We randomly partition, on each side of the biclique, the branch vertices in a balanced way, and consider the biclique subdivisions formed by every pair of parts (with one part per side). 
This way, every subdivision vertex is present in one induced subgraph, and we gain the desired control over the vertex-count drop.
The drop in the number of remaining extra edges works since the minimality condition implies that at~least three branch vertices are ``involved'' in locating an extra edge.
Therefore, for an extra edge to survive in a~considered induced subgraph, at least two branch vertices on the same side of the bipartition should land in the same part.
As a~result, in expectation, the number of extra edges overall contained in a~smaller biclique subdivision drops by at~least a~multiplicative factor equal to the number of parts on each side. 
We finally get the following detailed form of~\cref{thm:main}.

\begin{theorem}\label{thm:spec-main}
  For any natural numbers $t$ and $w$, and real $\varepsilon > 0$, there is an integer $W := W(t,w,\varepsilon)$ such that every graph with treewidth at least $W$ and no $K_{t,t}$ subgraph admits an $n$-vertex induced subgraph with treewidth at~least $w$ and at~most $(1+\varepsilon)n$ edges that is either
  \begin{compactitem}
  \item (the line graph of) a~subdivision of the $w \times w$ wall, or
  \item a~spanning supergraph of a~$K_{w,w}$ subdivision.
  \end{compactitem}
\end{theorem}

Technically, our proof of~\cref{thm:main} only yields for the first item a~quasi-subdivision\footnote{See~\cref{subsec:subd} for a~definition.} of the $w \times w$ wall.
However it is known that, by increasing $W$ accordingly and applying Ramsey's theorem, an induced subdivision of the $w \times w$ wall or the line graph of such a~subdivision can be found in a~quasi-subdivision of a~larger wall; see for instance~\cite[Lemma~3.6]{Aboulker21} where wall quasi-subdivisions are called \emph{stone walls}.

In particular, \cref{thm:spec-main} generalizes a~result by Weißauer~\cite{Weissauer19} stating that weakly sparse graphs of large treewidth have long induced cycles.
Indeed, quasi-subdivisions of large walls contain long induced cycles, and to possess only $(1+\varepsilon)n$ edges a~spanning supergraph of a~$K_{w,w}$ subdivision needs a(n induced) path of $\Theta(1/\varepsilon)$ vertices of \emph{overall} degree~2, implying the existence of an induced cycle of length $\Omega(1/\varepsilon)$.
As a~compromise to the now-refuted conjecture of Hajebi that graphs satisfying the conditions of~\cref{thm:main} admit 2-degenerate induced subgraphs of large treewidth~\cite{Hajebi22,Hajebi24}, \cref{thm:main} at~least ensures, for every $\varepsilon > 0$, the existence of induced subgraphs of large treewidth wherein all but an $\varepsilon$ fraction of the vertices have degree~2.

We conclude this section with a~summary of how low can be made the maximum degree/degeneracy/edge density/average degree of a~2-connected induced subgraph with treewidth at~least~$w$ found in graphs of arbitrarily large treewidth and no $K_{t,t}$ subgraph; see~\cref{tbl:summary}.
For the edge density and average degree, \cref{thm:main} gives universal, tight values.
No universal value can work for the maximum degree~\cite{Pohoata14,Davies22} and degeneracy~\cite{Chudnovsky24}.

\begin{table}[h!]
\begin{tabular}{lccc}
  \toprule
   & upper bound & question & lower bound \\
  \midrule
  maximum degree & & $O_t(1) \cdot w$? & any $O_t(1) \cdot o(w)$ \cite{Pohoata14,Davies22} \\
  degeneracy & some $O_w(1)$ \cite{Kuhn04,Girao24,Bourneuf23} & $O_t(1)$? & any $O(1)$~\cite{Chudnovsky24} \\
  edge density & $1+\varepsilon,~\forall \varepsilon > 0$~~(\cref{thm:main}) & $1+O_{t,w}(1) n^{-0.01}$? & 1 \\
  average degree & $2+\varepsilon,~\forall \varepsilon > 0$~~(\cref{thm:main}) & & 2 \\
  \bottomrule
\end{tabular}
\caption{Upper and lower bounds $\beta$ on sparsity parameters $\kappa$ such that graphs of arbitrarily high treewidth with some $K_{t-1,t-1}$ subgraph but no $K_{t,t}$ subgraph guarantee (resp. cannot guarantee) 2-connected ($n$-vertex) induced subgraphs~$H$ of treewidth at least~$w$, with $\kappa(H) \leqslant \beta$.}
\label{tbl:summary}
\end{table}

At first sight, the upper bounds for degeneracy (given by~\cite{Kuhn04,Girao24,Bourneuf23}) are some functions of $w$ \emph{and}~$t$.
However, if $t \leqslant w$, then $O_{w,t}(1) = O_w(1)$, and if instead $t > w$, then any subgraph induced by $2w$ vertices with $w$ vertices picked on each side of some $K_{t-1,t-1}$ subgraph has the requested properties.
A~possible interpretation of~\cref{thm:main,thm:spec-main} is that in addition to $n$ edges necessary to make an $n$-vertex induced subgraph $2$-connected (or $n-1$ edges to simply make it connected), we rely on $o(n)$ additional edges to increase the treewidth.
Can this $o(n)$ be replaced by, say, $O_{t,w}(1) \cdot n^{0.99}$?

\section{Preliminaries}

If $i$ is a~positive integer, we denote by $[i]$ the set of integers $\{1,2,\ldots,i\}$.

\subsection{Subgraphs, induced subgraphs, neighborhoods}

We denote by $V(G)$ and $E(G)$ the set of vertices and edges of a graph $G$, respectively.
A~graph $H$ is a~\emph{subgraph} of a~graph $G$ if $H$ can be obtained from $G$ by vertex and edge deletions.
Graph~$H$ is an~\emph{induced subgraph} of $G$ if $H$ is obtained from $G$ by vertex deletions only.
For $S \subseteq V(G)$, the \emph{subgraph of $G$ induced by $S$}, denoted $G[S]$, is obtained by removing from $G$ all the vertices that are not in $S$ (together with their incident edges).
Then $G-S$ is a short-hand for $G[V(G)\setminus S]$.
A~set $X \subseteq V(G)$ is connected (in $G$) if $G[X]$ has a~single connected component. 

We will often refer to two dense parameterized graphs: the \emph{$t$-clique}, denoted by \emph{$K_t$}, obtained by making adjacent every pair of two distinct vertices over $t$ vertices, and the \emph{biclique $K_{t,t}$} with bipartition $(A,B)$ such that $|A|=|B|=t$ obtained by making every vertex of~$A$ adjacent to every vertex of~$B$.   

We denote by $N_G(v)$ and $N_G[v]$, the open, respectively closed, neighborhood of $v$ in $G$.
For $S \subseteq V(G)$, we set $N_G(S) := (\bigcup_{v \in S}N_G(v)) \setminus S$ and $N_G[S] := N_G(S) \cup S$.
The \emph{degree} $d_G(v)$ of a~vertex $v \in V(G)$ is the size of $N_G(v)$, and the \emph{maximum degree} of $G$, denoted by~$\Delta(G)$, is $\max_{v \in V(G)} d_G(v)$.
A~\emph{subcubic graph} is a~graph of maximum degree at most~3.
The \emph{minimum degree} of $G$ is $\min_{v \in V(G)} d_G(v)$.

In all the previous notations, we may omit the graph subscript if it is clear from the context.
Finally, a~graph $G$ on at~least three vertices is said \emph{2-connected} if it has no vertex whose deletion disconnects~$G$.

\subsection{Tree-decompositions, adhesions, brambles}\label{sec:tree-dec-bramble}

A~\emph{tree-decomposition} of a~graph $G$ is a~pair $(T,\beta)$ where $T$ is a~tree and $\beta$ is a~map from $V(T)$ to $2^{V(G)}$ satisfying the following conditions:
\begin{compactitem}
\item for every $uv \in E(G)$, there is an~$x \in V(T)$ such that $\{u,v\} \subseteq \beta(x)$, and
\item for every $v \in V(G)$, the set of nodes $x \in V(T)$ such that $v \in \beta(x)$ induces a~non-empty subtree of $T$.
\end{compactitem}
The \emph{width} of $(T,\beta)$ is defined as $\max_{x \in V(T)} |\beta(x)| - 1$, and the \emph{treewidth} of $G$, denoted by $\tw(G)$, is the minimum width of $(T,\beta)$ taken among every tree-decomposition $(T,\beta)$ of~$G$.

An \emph{adhesion} of $(T,\beta)$ is any non-empty intersection $\beta(x) \cap \beta(y)$ with $x \neq y \in V(T)$.
We call \emph{adhesion of~$x$} any non-empty intersection $\beta(x) \cap \beta(y)$ with $y \in V(T) \setminus \{x\}$.
The \emph{adhesion size} of $(T,\beta)$ is defined as $\max_{x \neq y \in V(T)} |\beta(x) \cap \beta(y)|$. 

The notion of~\emph{bramble} was introduced by Seymour and Thomas~\cite{SeymourT93} as a~min-max dual to treewidth.
A~\emph{bramble} of a graph~$G$ is a~set $\mathcal B := \{B_1, \ldots, B_q\}$ of connected subsets of $V(G)$ such that for every $i, j \in [q]$ the pair $B_i, B_j$ \emph{touch}, i.e., $B_i \cap B_j \neq \emptyset$ or there is some $u \in B_i$ and $v \in B_j$ with $uv \in E(G)$.
A~\emph{hitting set} of~$\{B_1, \ldots, B_q\}$ is a~set $X$ such that for every $i \in [q]$, $X \cap B_i \neq \emptyset$, that is, $X$ intersects every $B_i$.
The \emph{order} of bramble $\mathcal B$ is the minimum size of a~hitting set of $\mathcal B$.

The \emph{bramble number} of $G$, denoted by $\bn(G)$, is the maximum order of a~bramble of~$G$.
The treewidth and bramble number are tied: For every graph $G$, $\tw(G)=\bn(G)-1$~\cite{SeymourT93}.

\subsection{Minors, induced minors, branch sets}

The \emph{contraction} of an edge $uv$ in a~graph $G$ results in~a graph $G'$ with $V(G')=(V(G) \setminus \{u,v\}) \cup \{w\}$ and $E(G')=E(G-\{u,v\}) \cup \{wx~:~x \in N_G(\{u,v\})\}$.
A~graph $H$ is a~\emph{minor} of a~graph $G$ if $H$ can be obtained from $G$ by edge contractions, and vertex and edge deletions.
It is an \emph{induced minor} if it is obtained by edge contractions and vertex deletions (but no edge deletions).

Minors and induced minors can equivalently be defined via \emph{(induced) minor models}.
A~\emph{minor model} of $H$ in~$G$ is a~collection $\mathcal M := \{B_1, \ldots, B_{|V(H)|}\}$ of pairwise-disjoint connected subsets of $V(G)$, called \emph{branch sets}, together with a~bijective map $\phi: V(H) \to \mathcal M$ such that if $uv \in E(H)$, then there is at~least one edge in $G$ between $\phi(u)$ and $\phi(v)$.
In which case, we may say that the branch sets $\phi(u)$ and $\phi(v)$ are \emph{adjacent}.

An~\emph{induced minor model} is similar with the stronger requirement that $uv \in E(H)$ if and only if there is at~least one edge in $G$ between $\phi(u)$ and $\phi(v)$.
An (induced) minor model $(\{B_1, \ldots, B_h\}, \phi)$ of an $h$-vertex graph $H$ is \emph{minimal} if for every $B'_1 \subseteq B_1, \ldots, B'_h \subseteq B_h$, the fact that $(\{B'_1, \ldots, B'_h\}, \phi')$ is an (induced) minor model of $H$ with $\phi'(u) = B'_i \Leftrightarrow \phi(u) = B_i$ (for every $u \in V(H)$) implies that for every $i \in [h]$, $B'_i = B_i$.
Note that we will mainly work with minor models of cliques, for which both the \emph{non-induced/induced} distinction and the relevance of $\phi$ vanish. 

Note that, as with subgraphs and induced subgraphs, \emph{being a minor of} and \emph{being an induced minor of} are transitive relations.
We will often use this fact, and mainly with induced subgraphs and induced minors.

\subsection{Subdivisions, quasi-subdivisions, topological minors}\label{subsec:subd}

A~\emph{subdivision} of a~graph $H$ is a graph $G$ obtained by replacing each edge of~$H$ by a~path on at~least one edge.
For a~natural number $\ell$, a~\emph{$(\geqslant \ell)$-subdivision} (resp.~\emph{$(\leqslant \ell)$-subdivision}) is obtained by replacing each edge of~$H$ by a~path on at~least $\ell$ edges (resp. at~least one and at~most $\ell$ edges).

Two distinct edges are \emph{incident} if they share an endpoint.
The \emph{line graph} of a~graph $G$ is a~graph with vertex set $E(G)$ and an edge between two vertices if and only if they correspond to incident edges in~$G$. 
A~\emph{tripod} is a~subdivision of $K_{1,3}$, the \emph{star} with three leaves.

If $H$ is a~subcubic graph, a~\emph{quasi-subdivision} of~$H$ is any graph obtained by replacing some vertices of degree~3 by triangles in a~subdivision of $H$.
More precisely, one may replace any vertex $v$ with three neighbors $x,y,z$ in the subdivision of $H$, by a triangle $v_xv_yv_z$ and make $v_x$ adjacent to $x$, $v_y$ adjacent to $y$, and $v_z$ adjacent to $z$.
(Note that the process is not iterative: one is not allowed to further replace a~vertex of a~created triangle by another triangle.)
For instance, the line graph of a $(\geqslant 2)$-subdivision of a~subcubic graph $H$ is a~particular quasi-subdivision of~$H$ (where every vertex of degree 3 was turned into a~triangle).

A~graph $H$ is a~\emph{topological minor} of~$G$ if $G$ contains a~subdivision of~$H$ as a~subgraph.
Topological minors are sometimes called \emph{topological subgraphs}.

\subsection{Degeneracy and edge density}

The \emph{degeneracy} of an $n$-vertex graph $G$ is the least integer~$d$ such that $V(G)$ can be linearly ordered $v_1, \ldots, v_n$ such that for every $i \in [n]$, $v_i$ has at~most $d$~neighbors in $v_{i+1}, \ldots, v_n$.
Note that an $n$-vertex graph with degeneracy at~most~$d$ has at~most $dn$ edges.

In this paper, the \emph{edge density} of a~graph $G$ is the value $|E(G)|/|V(G)|$, which is half of the average degree.
Thus the edge density is upper bounded by the degeneracy.
Conversely, there are graphs of arbitrarily large degeneracy and arbitrarily low edge density: for instance, the disjoint union of a~$t$-clique and $t^3$ isolated vertices, whose degeneracy is~$t-1$ and edge density is less than $1/t$.

\subsection{Shallow minors and expansion}

The \emph{radius $\text{rad}(G)$} of a graph $G$ is defined as $\min_{u \in V(G)} \max_{v \in V(G)} d_G(u,v)$, where $d_G(u,v)$ is the number of edges in a shortest path between $u$ and $v$.
The \emph{radius $\text{rad}_G(S)$} of a subset of vertices $S \subseteq V(G)$ is simply defined as $\text{rad}(G[S])$.
Note that two vertices can be further away in $G[S]$ than in $G$.
A~\emph{depth-$r$ minor} $H$ of $G$, denoted by $H \preccurlyeq_r G$, is a minor of~$G$ with branch sets $B_1, \ldots, B_{\card{V(H)}}$ satisfying $\text{rad}_G(B_i) \leqslant r$ for every $i \in [\card{V(H)}]$.
In particular depth-0 minors correspond to subgraphs.
The theory of graph sparsity pioneered by Nešetřil and Ossona de Mendez~\cite{sparsity} introduces the following invariants for a~graph $G$, for every $r \in \mathbb N$:
$$\nabla_r(G) := \underset{H \preccurlyeq_r G}{\sup}~\frac{|E(H)|}{|V(H)|}.$$
We say that a graph \emph{$G$ has expansion $f$} if $\nabla_r(G) \leqslant f(r)$ for every $r \in \mathbb N$.

\subsection{Grids and walls}

For two positive integers $k, \ell$, the \emph{$k \times \ell$ grid} is the graph on $k\ell$ vertices, say, $v_{i,j}$ with $i \in [k], j \in [\ell]$, such that $v_{i,j}$ and $v_{i',j'}$ are adjacent whenever either $i=i'$ and $|j-j'|=1$ or $j=j'$ and $|i-i'|=1$.
For what comes next, it is helpful to identify vertex $v_{i,j}$ with the point $(i,j)$ of $\mathbb N^2$.
In this paper, for $k \geqslant 2$ the \emph{$k \times k$ wall} is the subgraph of the $2k \times k$ grid obtained by removing every ``vertical edge'' on an ``even column'' when the edge bottom endpoint is on an ``odd row'', and every ``vertical edge'' on an ``odd column'' when the edge bottom endpoint is on an ``even row,'' and finally by deleting the two vertices of~degree~1 that this process creates.   
See~\cref{fig:grids-walls} for an illustration of the $5 \times 5$ grid and $5 \times 5$ wall.
\begin{figure}[h!]
  \centering
  \begin{tikzpicture}[vertex/.style={draw,circle,inner sep=0.035cm}]
    \def\t{5}
    \pgfmathtruncatemacro\tm{\t-1}
    \pgfmathtruncatemacro\tt{2 * \t}
    \pgfmathtruncatemacro\ttm{\tt - 1}
    \pgfmathtruncatemacro\ht{\t / 2}
    \def\s{0.6}

    \foreach \i in {1,...,\t}{
      \foreach \j in {1,...,\t}{
        \node[vertex] (v\i\j) at (\i * \s, \j * \s) {} ;
      }
    }
    \foreach \i in {1,...,\t}{
      \foreach \j [count = \jm from 1] in {2,...,\t}{
        \draw (v\i\j) -- (v\i\jm) ;
        \draw (v\j\i) -- (v\jm\i) ;
      }
    }

    \begin{scope}[xshift=1.4 * \t * \s cm]
    \foreach \i in {1,...,\tt}{
      \foreach \j in {2,...,\tm}{
        \node[vertex] (w\i-\j) at (\i * \s, \j * \s) {} ;
      }
    }
    \foreach \i in {1,...,\ttm}{
      \node[vertex] (w\i-1) at (\i * \s, \s) {} ;
    }
    \foreach \i in {2,...,\tt}{
      \node[vertex] (w\i-\t) at (\i * \s, \t * \s) {} ;
    }

    \foreach \i in {2,...,\tm}{
      \foreach \j [count = \jm from 1] in {2,...,\tt}{
        \draw (w\j-\i) -- (w\jm-\i) ;
      }
    }
    \foreach \j [count = \jm from 1] in {2,...,\ttm}{
        \draw (w\j-1) -- (w\jm-1) ;
    }
    \foreach \j [count = \jm from 2] in {3,...,\tt}{
        \draw (w\j-\t) -- (w\jm-\t) ;
    }

    \foreach \i in {1,...,\t}{
      \pgfmathtruncatemacro\ii{2 * \i}
      \foreach \j in {1,...,\ht}{
        \pgfmathtruncatemacro\jj{2 * \j}
        \pgfmathtruncatemacro\jjp{\jj + 1}
        \draw (w\ii-\jj) -- (w\ii-\jjp) ;
      }
    }
     \foreach \i in {1,...,\t}{
      \pgfmathtruncatemacro\ii{2 * \i - 1}
      \foreach \j in {1,...,\ht}{
        \pgfmathtruncatemacro\jj{2 * \j - 1}
        \pgfmathtruncatemacro\jjp{\jj + 1}
        \draw (w\ii-\jj) -- (w\ii-\jjp) ;
      }
    }
    \end{scope}
    
  \end{tikzpicture}
  \caption{The $5 \times 5$ grid (left) and the $5 \times 5$ wall (right).}
  \label{fig:grids-walls}
\end{figure}
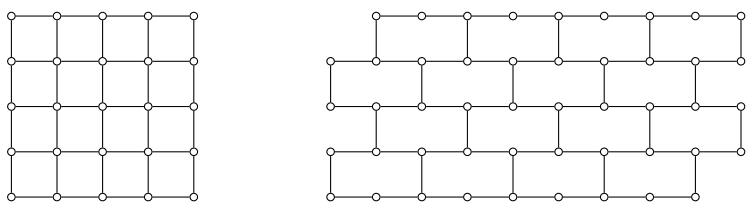
We may denote the $k \times k$ grid by $\Gamma_k$, and the $k \times k$ wall by $W_k$.

\section{Preparatory lemmas: grid and wall induced minors}

The following lemmas have been observed several times (see for instance~\cite{Aboulker21,FominGT11}).

\begin{lemma}\label{lem:grid-eq-wall}
  As induced minors, grids and walls are linearly tied in the sense that for every graph, the containment as an induced minor of  
  \begin{compactitem}
  \item $\Gamma_k$ implies that of $W_{\lfloor k/2 \rfloor}$\emph{;}
  \item $W_k$ implies that of $\Gamma_k$.
  \end{compactitem}
  Furthermore, any graph with $W_k$ as an induced minor contains a~quasi-subdivision of $W_{\lfloor k/3 \rfloor}$ as an induced subgraph.
\end{lemma}
\begin{proof}
  \Cref{fig:wk-gk} is a~visual proof of the two items. 
  \begin{figure}[h!]
  \centering
  \begin{tikzpicture}[vertex/.style={draw,circle,inner sep=0.035cm}]
    \def\t{8}
    \def\s{0.4}

    \foreach \i in {1,...,\t}{
      \foreach \j in {1,...,\t}{
        \node[vertex] (v\i\j) at (\i * \s, \j * \s) {} ;
      }
    }
    \foreach \i in {1,...,\t}{
      \foreach \j [count = \jm from 1] in {2,...,\t}{
        \draw (v\i\j) -- (v\i\jm) ;
        \draw (v\j\i) -- (v\jm\i) ;
      }
    }

    \foreach \i/\j in {2/1,4/1,6/1,1/3,3/3,5/3,2/5,4/5,6/5}{
      \pgfmathtruncatemacro\ip{\i+1}
      \pgfmathtruncatemacro\ipp{\ip+1}
      \pgfmathtruncatemacro\jp{\j+1}
      \pgfmathtruncatemacro\jpp{\jp+1}
      \draw[very thick, blue] (v\i\j) -- (v\ip\j) -- (v\ipp\j) -- (v\ipp\jp) -- (v\ipp\jpp) -- (v\ip\jpp) -- (v\i\jpp) -- (v\i\jp) -- (v\i\j)  ;
      \node[draw,blue,rectangle,fill opacity=0.05, fill,inner sep=0.08cm,rounded corners,fit = (v\i\jp) (v\i\jpp)] {} ;
    }

    \foreach \i/\j in {6/1,5/3,6/5}{
      \pgfmathtruncatemacro\ip{\i+1}
      \pgfmathtruncatemacro\ipp{\ip+1}
      \pgfmathtruncatemacro\jp{\j+1}
      \pgfmathtruncatemacro\jpp{\jp+1}
      \node[draw,blue,rectangle,fill opacity=0.05, fill,inner sep=0.08cm,rounded corners,fit = (v\ipp\jp) (v\ipp\jpp)] {} ;
    }

    \begin{scope}[xshift=1.4 * \t * \s cm, yshift=0.2 cm]
     \def\t{7}
    \pgfmathtruncatemacro\tm{\t-1}
    \pgfmathtruncatemacro\tt{2 * \t}
    \pgfmathtruncatemacro\ttm{\tt - 1}
    \pgfmathtruncatemacro\ht{\t / 2}

    \foreach \i in {1,...,\tt}{
      \foreach \j in {2,...,\tm}{
        \node[vertex] (w\i\j) at (\i * \s, \j * \s) {} ;
      }
    }
    \foreach \i in {1,...,\ttm}{
      \node[vertex] (w\i1) at (\i * \s, \s) {} ;
    }
    \foreach \i in {2,...,\tt}{
      \node[vertex] (w\i\t) at (\i * \s, \t * \s) {} ;
    }

    \foreach \i in {2,...,\tm}{
      \foreach \j [count = \jm from 1] in {2,...,\tt}{
        \draw (w\j\i) -- (w\jm\i) ;
      }
    }
    \foreach \j [count = \jm from 1] in {2,...,\ttm}{
        \draw (w\j1) -- (w\jm1) ;
    }
    \foreach \j [count = \jm from 2] in {3,...,\tt}{
        \draw (w\j\t) -- (w\jm\t) ;
    }

    \foreach \i in {1,...,\t}{
      \pgfmathtruncatemacro\ii{2 * \i}
      \foreach \j in {1,...,\ht}{
        \pgfmathtruncatemacro\jj{2 * \j}
        \pgfmathtruncatemacro\jjp{\jj + 1}
        \draw (w\ii\jj) -- (w\ii\jjp) ;
      }
    }
     \foreach \i in {1,...,\t}{
      \pgfmathtruncatemacro\ii{2 * \i - 1}
      \foreach \j in {1,...,\ht}{
        \pgfmathtruncatemacro\jj{2 * \j - 1}
        \pgfmathtruncatemacro\jjp{\jj + 1}
        \draw (w\ii\jj) -- (w\ii\jjp) ;
      }
     }

     \pgfmathsetmacro\h{0.35 * \s}

     \foreach \j in {1,...,\t}{
       \foreach \i in {1,3,5,7,...,\tt}{
      \draw[blue, fill opacity=0.05, fill, rounded corners] (\s * \i - \h,\s * \j - \h) -- (\s * \i - \h,\s * \j + \h) -- (\s * \i + \s + \h,\s * \j + \h) -- (\s * \i + \s + \h,\s * \j - \h) -- cycle ;
       }
     }
    \end{scope}
  \end{tikzpicture}
  \caption{Proof by picture of the two items of~\cref{lem:grid-eq-wall}:
    $\Gamma_8$ admits a~$W_4$ induced minor (left) and $W_7$ admits a~$\Gamma_7$ induced minor (right).
    The shaded boxes represent the contractions to perform.
    The vertices not incident to a~blue edge (left) should be deleted.}
  \label{fig:wk-gk}
  \end{figure}
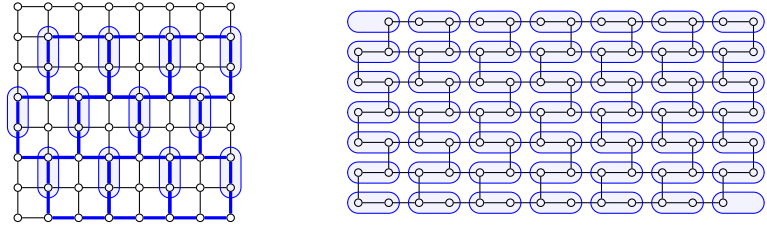
  Let $G$ be a~graph containing $W_k$ as an induced minor.
  Observe that $G$ then contains $H$, a~$(\geqslant 2)$-subdivision of $W_{\lfloor k/3 \rfloor}$, as an induced minor. 
  Let $(\mathcal M, \phi)$ be a~minimal induced minor model of $H$ in~$G$.
  We remark that $H$ is a~subcubic graph of minimum degree~2, such that every vertex of degree~3 in $H$ has only neighbors of degree~2.

  It can be seen that every branch set $B$ of $\mathcal M$ induces a~path (possibly on a~single vertex), a~tripod, or the line graph of a~tripod.
  We call~\emph{terminal} any vertex in~$B$ whose deletion is not disconnecting $G[B]$.
  By minimality, any~terminal $w$ of $B$ is adjacent to a~branch set $B'$ that is non-adjacent to $B \setminus \{w\}$.
  We refer to the \emph{degree} of a~branch set as the number of other branch sets it is adjacent~to, and we may say that two branch sets $B, B'$ are neighbors if they are adjacent.

  If $G[B]$ is a~tripod or the line graph of a~tripod, then $B$ has exactly three terminals, and no non-terminal of $B$ can be adjacent to another branch set of $\mathcal M$.
  The latter property trivially holds when $|B|=1$ (and there is then exactly one terminal in $B$).
  In both cases, $B$~can be entirely kept to (locally) form the wall quasi-subdivision.

  Therefore let us assume that $B$ induces a~path on at~least two vertices.
  Thus $B$ has exactly two terminals $u$ and $v$, the extremities of the path $G[B]$.
  Again, if $B$ has only two neighbors, no non-terminal of $B$ is adjacent to another branch set, and $B$ can be entirely kept.
  So we further assume that $B$ has three neighbors.

  Let $B_u$ (resp.~$B_v$) a~branch set of $\mathcal M$ adjacent to $u$ (resp. to $v$) and no other vertex of $B$.
  Let $B'$ be the third neighbor of $B$.
  Let $x$ be the vertex closest to $u$ (possibly $u$) along $G[B]$ that is adjacent to $B'$.
  Similarly, let $y$ be the vertex closest to $v$ (possibly $v$) along $G[B]$ that is adjacent to $B'$. 
  If $x=y$, then we can entirely keep $B$ to build the wall quasi-subdivision.
  Let $P$ be the (possibly empty) subset of vertices between $x$ and $y$ along path $G[B]$.  

  By the definition of $H$, branch set $B'$ has degree~2.
  Hence $x$ and $y$ have the same unique neighbor $z$ in $B'$.
  We remove $P$ from our quasi-subdivision.
  Then either $x$ and $y$ are adjacent, and $xyz$ is a~triangle of the quasi-subdivision, or $P$ was non-empty.
  In both cases, we get (locally) a~quasi-subdivision of $H$ at the cost of potentially smoothing out~$B'$, that is, shrinking the $(\geqslant 2)$-subdivision in a~subdivision.
  Thus we eventually get a~quasi-subdivision of $W_{\lfloor k/3 \rfloor}$ as an induced subgraph of~$G$.
\end{proof}

We observe that getting a~large grid or wall induced minor is sufficient for our purposes.

\begin{lemma}\label{lem:sparse-grid}
  For any natural number $w$ and real $\varepsilon > 0$, any graph having $\Gamma_k$ or $W_k$ as an induced minor with $k := \lfloor \frac{30 w}{\varepsilon} \rfloor$ contains a~2-connected $n$-vertex induced subgraph with treewidth at~least $w$ and at~most $(1+\varepsilon)n$ edges.  
\end{lemma}
\begin{proof}
  By~\cref{lem:grid-eq-wall}, any graph $G$ with a~$\Gamma_k$ or $W_k$ induced minor has an induced quasi-subdivision of $W_{\lfloor k/7 \rfloor}$.
  We extract an induced~quasi-subdivision of $H$ from the one of~$W_{\lfloor k/7 \rfloor}$ by skipping every other row, and only keeping every $(\lceil 4/\varepsilon \rceil + 1)$-st ``column;'' see~\cref{fig:wall-decr-density}.
  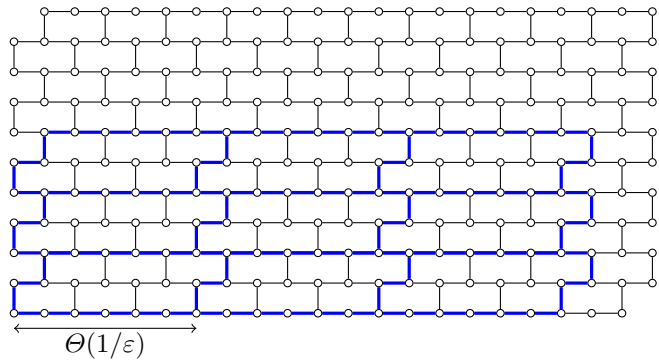
\begin{figure}[h!]
  \centering
  \begin{tikzpicture}[vertex/.style={draw,circle,inner sep=0.035cm}]
    \def\t{11}
    \pgfmathtruncatemacro\tm{\t-1}
    \pgfmathtruncatemacro\tt{2 * \t}
    \pgfmathtruncatemacro\ttm{\tt - 1}
    \pgfmathtruncatemacro\ht{\t / 2}
    \def\s{0.4}

    \foreach \i in {1,...,\tt}{
      \foreach \j in {2,...,\tm}{
        \node[vertex] (w\i-\j) at (\i * \s, \j * \s) {} ;
      }
    }
    \foreach \i in {1,...,\ttm}{
      \node[vertex] (w\i-1) at (\i * \s, \s) {} ;
    }
    \foreach \i in {2,...,\tt}{
      \node[vertex] (w\i-\t) at (\i * \s, \t * \s) {} ;
    }

    \foreach \i in {2,...,\tm}{
      \foreach \j [count = \jm from 1] in {2,...,\tt}{
        \draw (w\j-\i) -- (w\jm-\i) ;
      }
    }
    \foreach \j [count = \jm from 1] in {2,...,\ttm}{
        \draw (w\j-1) -- (w\jm-1) ;
    }
    \foreach \j [count = \jm from 2] in {3,...,\tt}{
        \draw (w\j-\t) -- (w\jm-\t) ;
    }

    \foreach \i in {1,...,\t}{
      \pgfmathtruncatemacro\ii{2 * \i}
      \foreach \j in {1,...,\ht}{
        \pgfmathtruncatemacro\jj{2 * \j}
        \pgfmathtruncatemacro\jjp{\jj + 1}
        \draw (w\ii-\jj) -- (w\ii-\jjp) ;
      }
    }
     \foreach \i in {1,...,\t}{
      \pgfmathtruncatemacro\ii{2 * \i - 1}
      \foreach \j in {1,...,\ht}{
        \pgfmathtruncatemacro\jj{2 * \j - 1}
        \pgfmathtruncatemacro\jjp{\jj + 1}
        \draw (w\ii-\jj) -- (w\ii-\jjp) ;
      }
     }

     \foreach \i/\j in {1/1,7/1,13/1, 1/3,7/3,13/3, 1/5,7/5,13/5}{
       \pgfmathtruncatemacro\jc{\j+2}
       \foreach \h in {0,...,5}{
         \pgfmathtruncatemacro\ip{\i+\h}
         \pgfmathtruncatemacro\ipp{\ip+1}
         \draw[draw,very thick,blue] (w\ip-\j) -- (w\ipp-\j) ;
       }
       \foreach \h in {1,...,6}{
         \pgfmathtruncatemacro\ip{\i+\h}
         \pgfmathtruncatemacro\ipp{\ip+1}
         \draw[draw,very thick,blue] (w\ip-\jc) -- (w\ipp-\jc) ;
       }
     }

     \foreach \i/\j in {1/1,7/1,13/1,19/1, 1/3,7/3,13/3,19/3, 1/5,7/5,13/5,19/5}{
       \pgfmathtruncatemacro\ip{\i+1}
       \pgfmathtruncatemacro\jp{\j+1}
       \pgfmathtruncatemacro\jpp{\jp+1}
       \draw[draw,very thick,blue] (w\i-\j) -- (w\i-\jp) -- (w\ip-\jp) -- (w\ip-\jpp) ;
     }

     \draw[<->] (\s,0.5 * \s) -- (7 * \s,0.5 *\s) ;
     \node at (4 * \s, - 0.4 * \s) {$\Theta(1/\varepsilon)$} ; 
  \end{tikzpicture}
  \caption{An induced quasi-subdivision (in blue) of large treewidth and edge density $(1+\varepsilon)$ in a~quasi-subdivision of the wall.
  Represented vertices can be vertices or triangles, edges are paths.}
  \label{fig:wall-decr-density}
  \end{figure}
  Let $H'$ be the induced subgraph of~$G$ corresponding to $H$.
  It can be seen that $H'$ is 2-connected, and $\tw(H') \geqslant \tw(H) \geqslant w$ since $H$ contains a~subdivision of the $w \times w$ wall.
  Notice that removing up to two edges of $G$ per vertex of degree 3 in $H$ (since this vertex can correspond to a~triangle in $G$) turns $H'$ into a~disjoint union of paths.
  Thus $|E(H')| \leqslant |V(H')|-1+2q$ where $q$ is the number of vertices of degree 3 in $H$.
  We conclude since $q$ is less than $\frac{\varepsilon}{2} |V(H')|$.
\end{proof}

\section{Finding a large clique as a topological minor}

\cref{lem:sparse-grid} implies that, should our graph contain a~large grid or wall induced minor, we would immediately reach our goal, the conclusion of~\cref{thm:main}.
We already observed that even among weakly sparse graphs, there are graphs of arbitrarily large treewidth avoiding a~constant-sized grid as an induced minor.
Nonetheless, we will show that such problematic graphs all contain a~large clique subdivision as a~subgraph.
This essentially consists of combining results by Korhonen~\cite{Korhonen23} and by Fomin, Golovach, and Thilikos~\cite{FominGT11}, respectively establishing the presence of large grids as induced minor when the maximum degree or Hadwiger number is small, with the structure theorem of graphs excluding a~topological minor by Grohe and Marx~\cite{GroheM15}.

We start by recalling the fairly recent, inspirational result of Korhonen. 

\begin{theorem}[\cite{Korhonen23}, Theorem 1]\label{thm:korhonen}
  There is a constant $c$ such that for any natural number $k$, every graph $G$ with treewidth at least $c \cdot k^{10} 2^{\Delta(G)^5}$ admits the~$k \times k$ grid as an induced minor.
\end{theorem}

Previously, the same conclusion was attained on graphs excluding a~fixed minor.

\begin{theorem}[\cite{FominGT11},Theorem 2]\label{thm:fgt11}
  For any graph $H$, there is a $c_H$ such that every $H$-minor-free graph of treewidth at~least $c_H \cdot k^2$ admits the~$k \times k$ grid as an induced minor. 
\end{theorem}

As~\cite[Theorem 2]{FominGT11} is phrased in terms of contraction (where vertex deletions are disallowed) rather than induced minor, the authors further have to require that the $H$-minor-free graph is connected.
Besides, the conclusion is that one obtains one of two parameterized graphs as a contraction: the triangulated grid or the triangulated grid plus one universal vertex.
\Cref{thm:fgt11} holds since the $k \times k$ triangulated grid (with or without an additional universal vertex) admits the~$k' \times k'$ grid as an induced minor, with $k' := \lfloor k/2 \rfloor - 1$.

Given a~tree-decomposition $(T,\beta)$ of a~graph $G$, and $x \in V(T)$, we define $G_x$ as the graph obtained from $G$ by contracting each connected component of $G-\beta(x)$ into a~single vertex, and then keeping only one representative per equivalence class of false twins among the contracted vertices. 

\begin{theorem}[\cite{GroheM15}, weaker statement than Theorem 4.1]\label{thm:grohe-marx}
  Let $H$ be a~$k$-vertex graph.
  There is a function $f: \mathbb N \to \mathbb N$ such that every graph $G$ excluding $H$ as a~topological minor admits a~tree-decomposition $(T,\beta)$ of adhesion size at~most $f(k)$ such that for every $x \in V(T)$, $G_x$~either excludes $K_{f(k)}$ as a~minor or has at~most $f(k)$ vertices of degree larger than $f(k)$.
\end{theorem}

In the actual statement of \cite[Theorem 4.1]{GroheM15}, Grohe and Marx distinguish five functions $a, b, c, d, e$ of $k$.
Four ($a,c,d,e$) correspond to the four occurrences of ``$f(k)$'' in our reformulation.
The fifth ($b$) is a~third outcome that $\beta(x)$ consists of at~most $b(k)$ vertices.
We set $g(k) := \max(a(k),b(k),c(k),d(k),e(k))$ and $f(k) := g(k)+2^{g(k)+1}$, and the latter case is thus assimilated to either of the two outcomes of~\cref{thm:grohe-marx}.
The reason for choosing $f$ in~\cref{thm:grohe-marx} instead of simply $g$ is clarified in the next paragraph, and in~\cref{lem:torso-to-gx}, which proves that our reformulation indeed holds. 

Another difference is that the authors show the stronger statement that the theorem holds for $\tau(x)$, the \emph{torso} of $x$, which is obtained from $G[\beta(x)]$ by turning every adhesion of $x$ into a~clique.
In comparison, our construction of $G_x$ does not add edges within $G[\beta(x)]$ but add an independent set $I$ such that every vertex of $I$ is adjacent to a~subset of an adhesion of $x$, and no pair of vertices of $I$ are false twins.
We now check that $G_x$ satisfies the conclusion of~\cref{thm:grohe-marx} (with function $f$) if $\tau(x)$ does so (with function $g$).

\begin{lemma}\label{lem:torso-to-gx}
  Let $G$ be a~graph with a~tree-decomposition $(T,\beta)$ of adhesion size at~most $h$, and $x \in V(T)$. 
  \begin{compactitem}
  \item If $\tau(x)$ excludes $K_h$ as a~minor, then $G_x$ excludes $K_{h+1}$ as a~minor.
  \item If $\tau(x)$ has at~most $h$ vertices of degree larger than $h$, then $G$ has at~most $h$ vertices with degree larger than~$h+2^{h+1}$.
  \end{compactitem}
\end{lemma}
\begin{proof}
  We still denote by $I$ the independent set $V(G_x) \setminus \beta(x)$.

  For the first item, assume that $G_x$ admits a~$K_{h+1}$ minor, and let $\{B_1,B_2,\ldots,B_{h+1}\}$ be a~minor model of $K_{h+1}$ in $G_x$.
  Observe that at~most one $B_i$ is a~singleton $\{v\}$ such that $v \in I$.
  Indeed, two such singletons would make a~non-edge since $I$ is an independent set.
  Therefore, without loss of generality assume that no branch set of $\{B_1,B_2,\ldots,B_h\}$ is a~single vertex in $I$.

  We claim that $\{B'_1 := B_1 \setminus I, B'_2 :=B_2 \setminus I, \ldots, B'_h := B_h \setminus I\}$ is a~minor model of $K_h$ in~$\tau(x)$.
  Every $B'_i$ is connected since the neighbors of any $v \in I$ in graph $G_x$ forms a clique in~$\tau(x)$.
  For the same reason and the fact that $I$ is an independent set, every pair $B'_i, B'_j$ (with $i \neq j \in [h]$) is adjacent in $\tau(x)$.

  For the second item, we first observe that every vertex $v \in I$ has degree at~most $h$ (in~$G_x$), since the neighborhood of~$v$ is contained in an adhesion of~$x$.
  Let $X \subseteq \beta(x)$ be the set of at~most~$h$ vertices with degree larger than~$h$ in~$\tau(x)$.
  We show that all the vertices of $V(G_x) \setminus X$ have degree at~most $h+2^{h+1}$ in~$G_x$.

  By the first observation, we can restrict ourselves to the vertices of $V(G_x) \setminus X$ in some adhesion of~$x$ (since the degree of the other vertices of $\beta(x)$ has not changed).
  Let $w$ be such a~vertex.
  As the degree of $w$ in $\tau(x)$ is at~most~$h$, there are at most $h$ vertices of $\beta(x)$ with which $w$ shares an adhesion of~$x$.
  Thus $w$ and these at~most $h$ vertices form a~set $Y \subseteq \beta(x)$ such that $|Y| \leqslant h+1$.
  Therefore, there are at~most $2^{h+1}$ vertices (actually $2^{h+1}-1$) of $I$ adjacent to~$w$.
  Recall indeed that only one vertex is kept in $I$ per equivalence class of false twins.
  Therefore, $w$ has degree at~most $h+2^{h+1}$ in $G_x$.
\end{proof}

We finally need the following lemma, which essentially implies that a~graph $G$ of large treewidth having a~tree-decomposition of small adhesion size admits a~$G_x$ of large treewidth.

\begin{lemma}\label{lem:bounded-adh-and-bag-tw}
  Let $h$ and $p$ be natural numbers, $G$ be a~graph of treewidth at~least~$hp$, and $(T,\beta)$ be a~tree-decomposition of~$G$ of adhesion size at~most~$h$.
  Then, there is an~$x \in V(T)$ such that $G_x$ has treewidth at least~$p$. 
\end{lemma}
\begin{proof}
  Recall, from the end of~\cref{sec:tree-dec-bramble}, that the treewidth of a~graph is equal to its bramble number minus~one.
  Let $\mathcal B := \{X_1, \ldots, X_q\}$ be a~bramble of order at~least~$hp+1$ in~$G$.
  For every $i \in [q]$, consider the subtree $T_i$ of $T$ induced by all the nodes $y \in V(T)$ such that $\beta(y) \cap X_i \neq \emptyset$.
  One can see that for $X_i, X_j$ to touch it should hold that $V(T_i), V(T_j)$ intersect.
  As subtrees of a~tree satisfy the Helly property, there is an~$x \in V(T)$ such that $x \in \bigcap_{i \in [q]} V(T_i)$.

  We now prove that $G_x$ has a~bramble $\mathcal B' := \{X'_1, \ldots, X'_q\}$ of order at~least~$p+1$. 
  We define each $X'_i$ from $X_i$ by following the construction of~$G_x$.
  That is, $X'_i$ is obtained from $X_i$ by contracting each connected component of $G-\beta(x)$ into a~single vertex.
  At this point, for every contracted vertex temporarily in $X'_i$, we put (permanently) in $X'_i$ the one representative $v \in I$ of its equivalence class of false twins, where $I$ is the independent set~$V(G_x) \setminus \beta(x)$.
  (The other contracted vertices of the equivalence class of $v$ are then discarded to get~$G_x$.)

  Note that this process creates a~connected~$X'_i$ in $G_x$.
  One can further notice that for every $i,j \in [q]$, sets $X'_i, X'_j$ touch in~$G_x$, since $X_i, X_j$ touch in~$G$.
  We shall finally argue that $\mathcal B'$ has no small hitting set.

  Assume for the sake of contradiction that $\mathcal B'$ has a~hitting set $Z$ of size at~most $p$.
  We build a~hitting set $Z'$ of~$\mathcal B$ in the following way.
  For every $z \in Z \cap \beta(x)$, we simply add $z$ to $Z'$, and for every $z \in Z \cap I$, we add to $Z'$ the at~most $h$ neighbors of $z$ in $G_x$.
  Thus $Z'$ is of size at~most~$hp$.
  We finally check that $Z'$ is indeed a~hitting set of~$\mathcal B$.
  For every $i \in [q]$, if $X'_i \cap Z \cap \beta(x) \neq \emptyset$, then $X_i \cap Z' (\cap \beta(x)) \neq \emptyset$.
  If instead $X'_i \cap Z \cap I \neq \emptyset$ and, say $z \in X'_i \cap Z \cap I$, then $X_i$ intersects a~connected component of $G-\beta(x)$ attached to $\beta(x)$ via $N_{G_x}(z)$.
  We conclude that $X_i$ intersects $Z'$ as, by assumption, $X_i$ also intersects~$\beta(x)$. 
\end{proof}

We can now show the main lemma of this section.
Observe that we do \emph{not} require here the absence of some biclique as a~subgraph.

\begin{replemma}{lem:clique-subdivision-intro} \label{lem:clique-subdivision}
  For any natural numbers $k$ and $s$, there is an integer $W := W(k,s)$ such that every graph of treewidth at~least $W$ either admits a~subdivision of the $s$-clique as a~subgraph or the $k \times k$ grid as an induced minor.
\end{replemma}
\begin{proof}
  Let $G$ be a~graph of treewidth at~least $W$ without $K_s$ as a~topological minor.
  We will show that $G$ admits $\Gamma_k$ as an induced minor (while specifying how $W(k,s)$ is chosen).

  By~\cref{thm:grohe-marx}, as $G$ excludes an $s$-vertex graph as a~topological minor (namely $K_s$), it~admits, for some function $f$, a~tree-decomposition $(T,\beta)$ of adhesion size at~most $f(s)$ such that for every $x \in V(T)$, $G_x$ excludes $K_{f(s)}$ as a~minor or has at~most $f(s)$ vertices of degree larger than $f(s)$.
  By~\cref{lem:bounded-adh-and-bag-tw}, choosing $W := f(s) \cdot W'$, there is at~least one $x \in V(T)$ such that $\tw(G_x) \geqslant W'$.

  We set $W' := \max(c \cdot k^{10} 2^{f(s)^5} + f(s), c_{K_{f(s)}} k^2)$ where $c$ is as in~\cref{thm:korhonen}, and $c_{K_{f(s)}}$ as in~\cref{thm:fgt11}.
  If~$G_x$ has no $K_{f(s)}$ minor, \cref{thm:fgt11} implies that $G_x$ admits the $k \times k$ grid as an induced minor.
  Importantly (and this is why we used $G_x$ instead of the torso of $x$), $G_x$ is an induced minor of $G$.
  Indeed, we obtained $G_x$ from $G$ by edge contractions followed by vertex deletions.
  Therefore, $G$ itself admits the $k \times k$ grid as an induced minor.
  
  If instead~$G_x$ has at~most $f(s)$ vertices of degree larger than $f(s)$, the graph $G_x$ deprived of these at~most $f(s)$ vertices has maximum degree at~most $f(s)$ and treewidth at least $c \cdot k^{10} 2^{f(s)^5} + f(s) - f(s) = c \cdot k^{10} 2^{f(s)^5}$, and we conclude similarly using~\cref{thm:korhonen}.
\end{proof}

\section{Upper bound on edge density, lower bound on subdivision length}

The following is a~celebrated result by Kühn and Osthus.

\begin{theorem}[\cite{Kuhn04}]\label{thm:kuhn-osthus}
  For any natural number $t$ and graph $H$, there is an integer $d := d(t,|V(H)|)$ such that every graph without $K_{t,t}$ subgraph nor induced subdivision of $H$ has degeneracy at~most $d$.
\end{theorem}

Dvořák showed that the upper bound on the degeneracy can be lifted to the expansion.

\begin{theorem}[\cite{Dvorak18}]\label{thm:dvorak}
  For any natural number $t$ and graph $H$, there is a function $f_{t,H}: \mathbb N \to \mathbb N$ such that every graph without $K_{t,t}$ subgraph nor induced subdivision of $H$ has expansion $f_{t,H}$.
\end{theorem}

Note that in~\cref{thm:dvorak}, we can freely assume that $f_{t,H}$ is~non-decreasing.
In light of these results and the previous section, we can get a~long topological minor of a~large clique with bounded degeneracy.
Before we make that formal, we need the following definitions.

A~\emph{spanning supergraph} $G'$ of a~graph $G$ is obtained from $G$ by adding a~possibly-empty subset of edges.
We refer to those added edges as the \emph{extra edges}, and we may denote their set by $E_{\text{extra}}(G') := E(G') \setminus E(G)$. 
Given a~natural number $\ell$ and a~graph $H$, a~$(\geqslant \ell)$-subdivision of $H$ is obtained by replacing every edge $e \in E(H)$ by a~path on at~least $\ell$ edges whose extremities are the endpoints of $e$.
We refer to those paths as \emph{direct paths}.
We call \emph{branch vertices} the vertices originally present in $H$, and \emph{subdivision vertices} the added vertices.
We keep the terminology of \emph{direct paths}, \emph{branch vertices}, and \emph{subdivision vertices} in a~spanning supergraph of the subdivision.

We say that a~spanning supergraph of a~subdivision is~\emph{trim} if each of its direct paths remains induced.
The trimness of a~spanning supergraph of a~$(\geqslant \ell)$-subdivision has the exact same definition.
So any chord on a~direct path contradicts the trimness even if it makes a~path of length less than $\ell$ between the two extremities of the direct path. 
Notice that, in particular, the branch vertices of a~trim spanning supergraph of a~$(\geqslant \ell)$-subdivision form an independent set when $\ell \geqslant 2$. 

\begin{lemma}\label{lem:long-clique-subd-bd-degen}
  For any natural numbers $t, w, s$, and real $\varepsilon > 0$, there are integers $W := W(t,s,w,\varepsilon)$, $\ell := \ell(t,s,w,\varepsilon) = g_{t,w,\varepsilon}(s)$ with $g_{t,w,\varepsilon}$ non-decreasing, $\lim_{s \to \infty} g_{t,w,\varepsilon}(s) = \infty$, and $d := d(t,w,\varepsilon)$ such that every graph without $K_{t,t}$ subgraph and with treewidth at~least~$W$ admits as an $n$-vertex induced subgraph one of the following
  \begin{compactitem}
  \item a~trim spanning supergraph of a~$(\geqslant \ell)$-subdivision of $K_s$ with at~most $dn$ edges, or
  \item a~2-connected graph with treewidth at least $w$ and at~most $(1 + \varepsilon)n$ edges.
  \end{compactitem}
\end{lemma}
\begin{proof}
  Setting $k := \lfloor \frac{30 w}{\varepsilon} \rfloor$ and $s' := {2s-2 \choose s-1}$, we apply \cref{lem:clique-subdivision-intro} with $W := W(k,s')$.
  If we get a~$\Gamma_k$ induced minor, we conclude by~\cref{lem:sparse-grid} since this lemma yields an induced subgraph meeting the conditions of the second item.
  We can thus assume that we obtain a~$K_{s'}$ subdivision as a~subgraph on some inclusion-wise minimal subset of vertices.
  Let $G$ be the subgraph induced by the vertices of this $K_{s'}$ subdivision.
  In particular, $G$ is a~trim spanning supergraph of a~$K_{s'}$ subdivision.
  
  Graph $G$ has no $K_{t,t}$ subgraph.
  We can further assume that $G$ does not have an induced subdivision of~$\Gamma_k$, since otherwise we conclude with~\cref{lem:sparse-grid}, as previously.
  Thus by~\cref{thm:dvorak}, $G$ has expansion at~most $f_{t,\Gamma_k}$.
  Let $\ell := \ell(t,s,w,\varepsilon)$ be the largest integer such that $f_{t,\Gamma_k}(\ell) < s-1$, or be equal to 0 if no such number exists.
  Consider the 2-edge-colored clique on the branch vertices of $G$ with a~blue edge between two branch vertices if the direct path linking them has at~most $\ell$~edges, and a~green edge otherwise.  
  By~Ramsey's theorem~\cite{Ramsey30} this auxiliary clique admits an all-green $s$-clique.
  Indeed an all-blue $s$-clique would contradict the expansion bound, as~$K_s$, of edge density $s-1$, would then be a~depth-$\ell$ minor.

  Let $G'$ be the trim spanning supergraph of a~$(\geqslant \ell)$-subdivision of $K_s$ induced by the branch vertices of the all-green $s$-clique together with all the direct paths between pairs of these particular branch vertices.
  In turn, $G'$ has no $\Gamma_k$ induced minor.
  Thus by~\cref{thm:kuhn-osthus}, $G'$ has degeneracy at~most $d := d(t,k^2) = d(t,w,\varepsilon)$, hence at~most $d|V(G')|$ edges.
  Therefore, the induced subgraph $G'$ satisfies the conditions of the first item.

  We finally need to check that, for every $t, w, \varepsilon$, $\lim_{s \to \infty} \ell(t,s,w,\varepsilon) = \infty$.
  This is immediate since $f_{t,\Gamma_k}$ is non-decreasing, and for any $\ell \in \mathbb N$, there is $s_\ell := f_{t,\Gamma_k}(\ell)+2$ such that $f_{t,\Gamma_k}(\ell) < s_\ell-1$. 
\end{proof}

\section{Decreasing the density of extra edges in the topological minor}

We now show how to decrease the ratio \emph{number of extra edges} over \emph{number of vertices} in spanning supergraphs of subdivisions.
A~technicality makes us switch from clique subdivisions to biclique subdivisions.
Setting $s := \frac{dw}{\varepsilon}$ and $h := \frac{s}{w}$ in the next lemma, this ratio goes from~$d$ down to at~most $\varepsilon$, while extracting a~$K_{w,w}$ subdivision from a~$K_{s,s}$ subdivision.

\begin{lemma}\label{lem:dedensifying}
  Let $G$ be a~trim spanning supergraph of a~$K_{s,s}$ subdivision $($resp.~$(\geqslant \ell)$-subdivision$)$ with $n$ vertices, and $m$ extra edges.
  For every $h$ dividing $s$, $G$ admits an induced subgraph $G'$ that is a~trim spanning supergraph of a~$K_{\frac{s}{h},\frac{s}{h}}$ subdivision $($resp.~$(\geqslant \ell)$-subdivision$)$ with $n'$ vertices and at~most $\frac{m}{hn} \cdot n'$ extra edges.
\end{lemma}
\begin{proof}
  Let $(A,B)$ be the bipartition of the branch vertices of $G$, with $|A|=|B|=s$.
  Let $(A_1, A_2, \ldots, A_h)$ be a~balanced (i.e., such that $|A_1| = |A_2| = \ldots = |A_h| = s/h$) partition of~$A$, and $(B_1, B_2, \ldots, B_h)$ be a~balanced partition of~$B$, taken independently uniformly at~random.
  For every $i,j \in [h]$, let $G_{i,j}$ be the subgraph of $G$ induced by the branch vertices in $A_i \cup B_j$ and the subdivision vertices in direct paths linking vertices of~$A_i$ to vertices of~$B_j$.
  Each graph $G_{i,j}$ is a~trim spanning supergraph of a~$K_{\frac{s}{h},\frac{s}{h}}$ subdivision.
  Further notice that if $G$ is a~trim spanning supergraph of a~$(\geqslant \ell)$-subdivision of~$K_{s,s}$, then each $G_{i,j}$ is a~trim spanning supergraph of a~$(\geqslant \ell)$-subdivision of~$K_{\frac{s}{h},\frac{s}{h}}$.

  Observe that, by trimness, an extra edge of $G$ is between two subdivision vertices on distinct direct paths, or between a branch vertex and a~subdivision vertex on a~direct path that is not incident to this branch vertex.
  Thus for an extra edge to be in some $G_{i,j}$, at least three branch vertices, say, $x, y, z$ have to be simultaneously present in $A_i \cup B_j$: the at~least three distinct extremities of the two direct paths in the former case, and the branch vertex plus the extremities of the direct path, in the latter.
  Among these three vertices, two, say $x, y$, need to be on the same side of the bipartition $(A,B)$.
  
  Thus, an extra edge of $G$ survives in at~least one $G_{i,j}$ with probability less than $1/h$, which slightly overestimates the probability that a~pair $x, y$ land in the same $A_i$ or $B_j$.
  Therefore the expected number of extra edges of $G$ still present in at~least one $G_{i,j}$ is smaller than $m/h$.
  By Markov's inequality, there is a~balanced partition $(A_1, A_2, \ldots, A_h)$ of~$A$ and a~balanced partition $(B_1, B_2, \ldots, B_h)$ of~$B$ such that \[\sum\limits_{i,j \in [h]} |E_{\text{extra}}(G_{i,j})| \leqslant \frac{m}{h}.\]
  Observe that the edge sets of the graphs $G_{i,j}$ are pairwise disjoint.
  Importantly, every subdivision vertex of $G$ is in one $G_{i,j}$ (and the branch vertices of $G$ are in $h$ graphs $G_{i,j}$).
  That was the reason to move from a~clique to a biclique subdivision.
  From \[\sum\limits_{i,j \in [h]} |V(G_{i,j})| \geqslant n~\text{and}~\sum\limits_{i,j \in [h]} |E_{\text{extra}}(G_{i,j})| = \left|\bigcup_{i,j \in [h]} E_{\text{extra}}(G_{i,j})\right| \leqslant \frac{m}{h},\] we deduce that there is one graph~$G_{i,j}$ with, say, $n'$ vertices and at~most $\frac{m}{hn} \cdot n'$ extra edges.
\end{proof}

\section{Wrapping up}

We can now prove our main theorem, which we repeat for convenience.

\begin{reptheorem}{thm:main}\label{thm:main-recalled}
  For any natural numbers $t$ and $w$, and real $\varepsilon > 0$, there is an integer $W := W(t,w,\varepsilon)$ such that every graph with treewidth at~least $W$ and no $K_{t,t}$ subgraph admits a~2-connected $n$-vertex induced subgraph with treewidth at~least $w$ and at~most $(1+\varepsilon)n$ edges.
\end{reptheorem}
\begin{proof}
  As showing the theorem for any sufficiently large $w$ actually implies the theorem, we assume that $w \geqslant 2$. 
  We apply~\cref{lem:long-clique-subd-bd-degen} on a~graph satisfying the premises of the theorem, setting $W:=W(t,2s,w,\varepsilon)$ for some well-chosen~$s$.
  As the second outcome of that lemma is our objective, we assume that we get an $n'$-vertex graph $G$ that is a~trim spanning supergraph of a~$(\geqslant \ell)$-subdivision of $K_{2s}$ with at~most $dn'$ edges, for $\ell := \ell(t,2s,w,\varepsilon) = g_{t,w,\varepsilon}(2s)$ and $d := d(t,w,\varepsilon)$.
  We choose $s$ as the smallest integer at~least equal to $\lceil \frac{4d}{\varepsilon} \rceil \cdot w$ and multiple of $\lceil \frac{4d}{\varepsilon} \rceil$ such that $g_{t,w,\varepsilon}(2s) > \frac{2}{\varepsilon}+1$.
  This is well-defined since $d$ depends on $t$, $w$, and~$\varepsilon$ only, $g_{t,w,\varepsilon}$ is non-decreasing, and $\lim_{s \to \infty} g_{t,w,\varepsilon}(2s) = \infty$.
  From our choice of $s$, it holds in particular that $\ell \geqslant \lceil \frac{2}{\varepsilon} \rceil + 1$.

  Let $X$ be the set of branch vertices of $G$, hence $|X|=2s$.
  Consider $(A,B)$, a~balanced (i.e., such that $|A|=|B|=s$) bipartition of $X$ chosen uniformly at random.
  Let $G'$ be the induced subgraph of $G$ obtained by removing all the subdivision vertices on direct paths between vertices on the same side of $(A,B)$ (i.e., between every pair of vertices in $A$ and every pair of vertices in $B$).
  Observe that, for every vertex $v$ of $G$, the probability that $v$ survives in $G'$ is larger than $1/2$.
  Indeed it is exactly 1 for a~branch vertex, and $s^2/{2s \choose 2}=s/(2s-1) > 1/2$ for a~subdivision vertex.
  Thus by Markov's inequality, there is a~balanced bipartition $(A,B)$ of $X$ such that $G'$ has at~least $n'/2$ vertices.
  We proceed with this induced subgraph $G'$.

  As $G$ has at~most $dn'$ edges, $G'$ has at~most $dn'$ extra edges.
  Therefore, the ratio $|E_{\text{extra}}(G')|/|V(G')|$ is at most $2d$.
  One can see that $G'$ is a~trim spanning supergraph of a~$(\geqslant \ell)$-subdivision of $K_{s,s}$.
  We thus apply~\cref{lem:dedensifying} on $G'$ with $h := \lceil 4d/\varepsilon \rceil \leqslant s/w$, and obtain an induced subgraph $G''$ that is a~trim spanning supergraph of a~$(\geqslant \ell)$-subdivision of $K_{\frac{s}{h},\frac{s}{h}}$ on $n$ vertices and at~most $\frac{\varepsilon n}{2}$ extra edges.

  $G''$ is our eventual graph.
  As a~spanning supergraph of a~subdivision of~$K_{\frac{s}{h},\frac{s}{h}}$ with $\frac{s}{h} \geqslant w \geqslant 2$, it is indeed 2-connected.
  It further contains $K_{\frac{s}{h},\frac{s}{h}}$ as a~minor, thus $\tw(G'') \geqslant \tw(K_{\frac{s}{h},\frac{s}{h}})= \frac{s}{h} \geqslant w$.
  It remains to check that any biclique $(\geqslant \ell)$-subdivision has edge density at~most $1+\frac{\varepsilon}{2}$.
  Then $|E(G'')|$ is indeed at~most $(1+\frac{\varepsilon}{2})n+\frac{\varepsilon n}{2}=(1+\varepsilon)n$, where $\frac{\varepsilon n}{2}$ accounts for the extra edges.

  We observe that the edge set of the subdivision is the union of the edge sets of its direct paths, whereas every subdivision vertex is by definition on exactly one direct path.
  Each direct path has some $p \geqslant \ell-1$ subdivision vertices and $p+1$ edges, thus the number of edges of the subdivision is at most~$(1+\frac{1}{\ell-1})n \leqslant (1+\frac{\varepsilon}{2})n$. 
\end{proof}


\begin{thebibliography}{10}

\bibitem{Aboulker21}
Pierre Aboulker, Isolde Adler, Eun~Jung Kim, Ni~Luh~Dewi Sintiari, and Nicolas
  Trotignon.
\newblock On the tree-width of even-hole-free graphs.
\newblock {\em Eur. J. Comb.}, 98:103394, 2021.
\newblock \href {https://doi.org/10.1016/j.ejc.2021.103394}
  {\path{doi:10.1016/j.ejc.2021.103394}}.

\bibitem{istd7}
Tara Abrishami, Bogdan Alecu, Maria Chudnovsky, Sepehr Hajebi, and Sophie
  Spirkl.
\newblock {Induced subgraphs and tree decompositions VII. Basic obstructions in
  H-free graphs}.
\newblock {\em J. Comb. Theory, Ser. {B}}, 164:443--472, 2024.
\newblock \href {https://doi.org/https://doi.org/10.1016/j.jctb.2023.10.008}
  {\path{doi:https://doi.org/10.1016/j.jctb.2023.10.008}}.

\bibitem{Bonamy23}
Marthe Bonamy, Édouard Bonnet, Hugues Déprés, Louis Esperet, Colin Geniet,
  Claire Hilaire, Stéphan Thomassé, and Alexandra Wesolek.
\newblock Sparse graphs with bounded induced cycle packing number have
  logarithmic treewidth.
\newblock {\em J. Comb. Theory, Ser. {B}}, 167:215--249, 2024.
\newblock \href {https://doi.org/https://doi.org/10.1016/j.jctb.2024.03.003}
  {\path{doi:https://doi.org/10.1016/j.jctb.2024.03.003}}.

\bibitem{Bourneuf23}
Romain Bourneuf, Matija Bucić, Linda Cook, and James Davies.
\newblock On polynomial degree-boundedness, 2023.
\newblock \href {http://arxiv.org/abs/2311.03341} {\path{arXiv:2311.03341}}.

\bibitem{Chekuri15}
Chandra Chekuri and Julia Chuzhoy.
\newblock Degree-3 treewidth sparsifiers.
\newblock In Piotr Indyk, editor, {\em Proceedings of the Twenty-Sixth Annual
  {ACM-SIAM} Symposium on Discrete Algorithms, {SODA} 2015, San Diego, CA, USA,
  January 4-6, 2015}, pages 242--255. {SIAM}, 2015.
\newblock \href {https://doi.org/10.1137/1.9781611973730.19}
  {\path{doi:10.1137/1.9781611973730.19}}.

\bibitem{Chudnovsky24}
Maria Chudnovsky and Nicolas Trotignon.
\newblock On treewidth and maximum cliques, 2024.
\newblock \href {http://arxiv.org/abs/2405.07471} {\path{arXiv:2405.07471}}.

\bibitem{Courcelle90}
Bruno Courcelle.
\newblock {The Monadic Second-Order Logic of Graphs. I. Recognizable Sets of
  Finite Graphs}.
\newblock {\em Inf. Comput.}, 85(1):12--75, 1990.
\newblock \href {https://doi.org/10.1016/0890-5401(90)90043-H}
  {\path{doi:10.1016/0890-5401(90)90043-H}}.

\bibitem{Davies22}
James Davies.
\newblock Oberwolfach report 1/2022.
\newblock doi:10.4171/OWR/2022/1., 2022.

\bibitem{Dvorak18}
Zdenek Dvor{\'{a}}k.
\newblock Induced subdivisions and bounded expansion.
\newblock {\em Eur. J. Comb.}, 69:143--148, 2018.
\newblock \href {https://doi.org/10.1016/j.ejc.2017.10.004}
  {\path{doi:10.1016/j.ejc.2017.10.004}}.

\bibitem{FominGT11}
Fedor~V. Fomin, Petr~A. Golovach, and Dimitrios~M. Thilikos.
\newblock Contraction obstructions for treewidth.
\newblock {\em J. Comb. Theory, Ser. {B}}, 101(5):302--314, 2011.
\newblock \href {https://doi.org/10.1016/j.jctb.2011.02.008}
  {\path{doi:10.1016/j.jctb.2011.02.008}}.

\bibitem{Girao24}
António Girão and Zach Hunter.
\newblock Induced subdivisions in $k_{s,s}$-free graphs with polynomial average
  degree, 2024.
\newblock \href {http://arxiv.org/abs/2310.18452} {\path{arXiv:2310.18452}}.

\bibitem{GroheM15}
Martin Grohe and D{\'{a}}niel Marx.
\newblock Structure theorem and isomorphism test for graphs with excluded
  topological subgraphs.
\newblock {\em {SIAM} J. Comput.}, 44(1):114--159, 2015.
\newblock \href {https://doi.org/10.1137/120892234}
  {\path{doi:10.1137/120892234}}.

\bibitem{Hajebi22}
Sepehr Hajebi.
\newblock Problem 13.
\newblock {\em Open problems for the second 2022 Barbados workshop,
  \emph{\url{https://web.math.princeton.edu/~tunghn/2022openproblems.pdf}}},
  2022.
\newblock URL:
  \url{https://web.math.princeton.edu/~tunghn/2022openproblems.pdf}.

\bibitem{Hajebi24}
Sepehr Hajebi.
\newblock Chordal graphs, even-hole-free graphs and sparse obstructions to
  bounded treewidth, 2024.
\newblock \href {http://arxiv.org/abs/2401.01299} {\path{arXiv:2401.01299}}.

\bibitem{Korhonen21}
Tuukka Korhonen.
\newblock A single-exponential time 2-approximation algorithm for treewidth.
\newblock In {\em 62nd {IEEE} Annual Symposium on Foundations of Computer
  Science, {FOCS} 2021, Denver, CO, USA, February 7-10, 2022}, pages 184--192.
  {IEEE}, 2021.
\newblock \href {https://doi.org/10.1109/FOCS52979.2021.00026}
  {\path{doi:10.1109/FOCS52979.2021.00026}}.

\bibitem{Korhonen23}
Tuukka Korhonen.
\newblock Grid induced minor theorem for graphs of small degree.
\newblock {\em J. Comb. Theory, Ser. {B}}, 160:206--214, 2023.
\newblock \href {https://doi.org/10.1016/j.jctb.2023.01.002}
  {\path{doi:10.1016/j.jctb.2023.01.002}}.

\bibitem{KorhonenL23}
Tuukka Korhonen and Daniel Lokshtanov.
\newblock Induced-minor-free graphs: Separator theorem, subexponential
  algorithms, and improved hardness of recognition.
\newblock In David~P. Woodruff, editor, {\em Proceedings of the 2024 {ACM-SIAM}
  Symposium on Discrete Algorithms, {SODA} 2024, Alexandria, VA, USA, January
  7-10, 2024}, pages 5249--5275. {SIAM}, 2024.
\newblock \href {https://doi.org/10.1137/1.9781611977912.188}
  {\path{doi:10.1137/1.9781611977912.188}}.

\bibitem{Kuhn04}
Daniela K{\"{u}}hn and Deryk Osthus.
\newblock Induced subdivisions in {K}\({}_{\mbox{s, s}}\)-free graphs of large
  average degree.
\newblock {\em Comb.}, 24(2):287--304, 2004.
\newblock \href {https://doi.org/10.1007/s00493-004-0017-8}
  {\path{doi:10.1007/s00493-004-0017-8}}.

\bibitem{sparsity}
Jaroslav Nesetril and Patrice~Ossona de~Mendez.
\newblock {\em Sparsity - Graphs, Structures, and Algorithms}, volume~28 of
  {\em Algorithms and combinatorics}.
\newblock Springer, 2012.
\newblock \href {https://doi.org/10.1007/978-3-642-27875-4}
  {\path{doi:10.1007/978-3-642-27875-4}}.

\bibitem{Pohoata14}
Andrei~Cosmin Pohoaţă.
\newblock Unavoidable induced subgraphs of large graphs.
\newblock 2014.

\bibitem{Ramsey30}
Frank~P. Ramsey.
\newblock On a problem of formal logic.
\newblock In {\em Proc. London Math. Soc. series 2}, volume~30 of {\em
  264-286}, 1930.

\bibitem{RobertsonS86}
Neil Robertson and Paul~D. Seymour.
\newblock {Graph minors. V. Excluding a planar graph}.
\newblock {\em J. Comb. Theory, Ser. {B}}, 41(1):92--114, 1986.
\newblock \href {https://doi.org/10.1016/0095-8956(86)90030-4}
  {\path{doi:10.1016/0095-8956(86)90030-4}}.

\bibitem{SeymourT93}
Paul~D. Seymour and Robin Thomas.
\newblock Graph searching and a min-max theorem for tree-width.
\newblock {\em J. Comb. Theory, Ser. {B}}, 58(1):22--33, 1993.
\newblock \href {https://doi.org/10.1006/jctb.1993.1027}
  {\path{doi:10.1006/jctb.1993.1027}}.

\bibitem{SintiariT21}
Ni~Luh~Dewi Sintiari and Nicolas Trotignon.
\newblock (theta, triangle)-free and (even hole, {K}\({}_{\mbox{4}}\))-free
  graphs - part 1: Layered wheels.
\newblock {\em J. Graph Theory}, 97(4):475--509, 2021.
\newblock \href {https://doi.org/10.1002/jgt.22666}
  {\path{doi:10.1002/jgt.22666}}.

\bibitem{Weissauer19}
Daniel Wei{\ss}auer.
\newblock In absence of long chordless cycles, large tree-width becomes a local
  phenomenon.
\newblock {\em J. Comb. Theory, Ser. {B}}, 139:342--352, 2019.
\newblock \href {https://doi.org/10.1016/j.jctb.2019.04.004}
  {\path{doi:10.1016/j.jctb.2019.04.004}}.

\end{thebibliography}
\end{document}